\documentclass[a4paper]{article}

\usepackage[utf8]{inputenc}

\usepackage{amsmath,amssymb,theorems}
\usepackage{mathtools}
\usepackage{tikz}
\usetikzlibrary{decorations.pathreplacing}

\usepackage[font=small,labelfont=bf]{caption}

\usepackage{hyperref}

\let\epsilon\varepsilon
\newcommand\diff[2]{\frac{\partial#1}{\partial#2}}
\newcommand\abs[1]{\lvert#1\rvert}
\newcommand\bigo{\mathcal O}

\newcommand\lw{lw}
\def\seq{\text{\textsc{Seq}}}

\tikzstyle{help lines}=[black!40, very thin]
\tikzstyle{grid}=[help lines]
\tikzstyle{animals}=[scale=.27]
\tikzstyle{dark gray}=[color=black!60]
\tikzstyle{box}=[semithick]
\tikzstyle{ground}=[semithick]
\tikzstyle{lattice}=[animals,>=stealth,semithick]
\tikzstyle{site}=[black,circle,fill,inner sep=0pt,minimum size=1.05mm]
\tikzstyle{arrow}=[-latex]

\newcommand\site{node [site] {}}
\newcommand\segment[3]{
\begin{scope}[xshift=#1cm,yshift=#2cm]
\begin{scope}[xshift=-1cm]
\foreach \x in {1,...,#3}
  \node at (\x,0) [site] {};
\end{scope}
\end{scope}}
\newcommand\polymer[3]{
\draw (#1,#2) \site -- ++(#3,0) \site;}

\newlength\eb
\newcommand\arrow[1]{%
\setlength{\eb}{\fontcharht\font`B}%
\tikz[x=\eb,y=\eb]\draw[-latex](0,0) -- ++(#1);}

\newcommand\equals{%
\hspace{2.5mm}%
\tikz[animals,baseline=0cm] \node at (0,2) {$=$};%
\hspace{2.5mm}%
}

\newcommand\plus{%
\hspace{2.5mm}%
\tikz[animals,baseline=0cm]\node at (0,2) {$+$};%
\hspace{2.5mm}%
}

\let\olditemize\itemize
\def\itemize{\olditemize\setlength\itemsep{0pt}}
\let\oldenumerate\enumerate
\def\enumerate{\oldenumerate\setlength\itemsep{0pt}}

\title{Directed and multi-directed animals\\in the king's lattice}

\begin{document}

\author{Axel Bacher\thanks{LIPN, Université Paris 13 ---
\texttt{bacher@lipn.fr}}}

\maketitle

\begin{abstract}
This article introduces a new, simple solvable lattice for directed animals: 
the \emph{directed king's lattice}, or square lattice with next nearest
neighbor bonds and preferred directions
\smash{$\{\arrow{-1,0},\arrow{-1,1},\arrow{0,1},\arrow{1,1},\arrow{1,0}\}$}.
We show that the directed animals in this lattice have an algebraic generating
function linked to the Schröder numbers and belong to the same universality
class as the ones in the square and triangular lattices. We also define
\emph{multi-directed animals} in the king's lattice, which form a superclass
of directed animals. We compute their generating function and show that it is
not D-finite.  Finally, we propose efficient random sampling algorithms for
our animals.

\bigskip


\bigskip

\end{abstract}

\section{Introduction} \label{sec:intro}

An animal in a lattice is a finite and connected set of vertices. The
enumeration of animals (up to a translation) is a longstanding problem in
statistical physics and combinatorics. The problem, however, is extremely
difficult, and little progress has been made \cite{rivest,guttmann-square}. A
more realistic goal, therefore, is to enumerate natural subclasses of animals.

The class of \emph{directed animals} is one of the most classical of these
subclasses, with connections to directed percolation problems.  Consider a
lattice with oriented arcs and let $s$ be a vertex of this lattice.  An animal
$A$ is \emph{directed} if, for every site $t$ of~$A$, there exists a directed
path from~$s$ to~$t$ visiting only sites of~$A$. Directed animals were first
enumerated in the square and triangular lattices
\cite{nadal,hakim,dhar,gouyou,betrema} (Figure~\ref{fig:lattices}). Their
generating function was found to be algebraic.

However, most two-dimensional lattices are still unsolved, including the
honeycomb lattice, where the generating function of directed animals is
believed to be non-algebraic \cite{guttmann}. A natural question is,
therefore, to find out which lattices are solvable. Currently solved lattices
include Bousquet-Mélou and Conway's lattices~$\mathcal L_n$
\cite{bousquet-conway,denise} and the ``strange'' or $n$-decorated lattices
\cite{brak,bousquet}, both of which are infinite families of lattices which
can be seen as extensions of the square lattice. To our knowledge, no other
lattice has been solved since then.

Another way to explore animal enumeration is to enumerate superclasses of the
directed animals. The \emph{multi-directed animals} form such a superclass,
defined by Bousquet-Mélou and Rechnitzer \cite{rechnitzer} in the square and
triangular lattices based on earlier work by Klarner \cite{klarner}. They gave
closed expressions for the generating functions of multi-directed animals and
showed that they are not D-finite.

This paper introduces a new lattice, defined as the square lattice with added
diagonal (next nearest neighbor) bonds and the orientations
\smash{$\{\arrow{-1,0}, \arrow{-1,1}, \arrow{0,1}, \arrow{1,1},
\arrow{1,0}\}$} (Figure~\ref{fig:lattices}, right). We call this lattice the
\emph{king's lattice} as it recalls the king's moves in chess. The directed
animals in the king's lattice are a superclass of the directed animals in
Bousquet-Mélou and Conway's lattice~$\mathcal L_3$, which has orientations
\smash{$\{\arrow{-1,1}, \arrow{0,1}, \arrow{1,1}\}$} \cite{bousquet-conway}.
Our goal is to study both directed and multi-directed animals in this lattice.

\begin{figure}[ht]
\begin{center}
\begin{tikzpicture}[animals,baseline=0pt]
\begin{scope}
\clip (-4.4,-.4) rectangle (4.4,5.4);
\draw [grid,cm={1,1,-1,1,(0,0)}] (0,0) grid (6,6);
\foreach \p in
{(0,0),(1,1),(0,2),(2,2),(-1,3),(3,3),(-2,4),(0,4),(2,4),(-3,5),(1,5),(3,5)}
    \node at \p [site] {};
\end{scope}
\draw[-latex] (0,-3) -- ++(-1,1);
\draw[-latex] (0,-3) -- ++(1,1);
\end{tikzpicture}\hfill%
\begin{tikzpicture}[animals,baseline=0pt]
\begin{scope}
\clip (-4.4,-.4) rectangle (4.4,5.4);
\draw [grid,cm={1,1,-1,1,(0,0)}] (0,0) grid (6,6);
\draw [grid] (0,0) -- ++(0,6);
\foreach \x in {1,...,4}
    \draw[grid] (-\x,\x) -- ++(0,6) (\x,\x) -- ++(0,6);
\foreach \p in
{(0,0),(-1,1),(1,1),(-2,2),(2,2),(1,3),(3,3),(-2,4),(-3,5),(1,5),(3,5)}
    \node at \p [site] {};
\end{scope}
\draw[-latex] (0,-3) -- ++(-1,1);
\draw[-latex] (0,-3) -- ++(0,2);
\draw[-latex] (0,-3) -- ++(1,1);
\end{tikzpicture}\hfill%
\begin{tikzpicture}[animals,baseline=0pt]
\begin{scope}
\clip (-4.4,-.4) rectangle (4.4,5.4);
\draw [grid,cm={1,1,-1,1,(0,0)}] (0,0) grid (6,6);
\draw [grid,cm={1,1,-1,1,(0,1)}] (0,0) grid (6,6);
\draw [grid] (0,0) -- ++(0,6);
\foreach \x in {1,...,5}
    \draw[grid] (-\x,\x) -- ++(0,6) (\x,\x) -- ++(0,6);
\foreach \p in
{(0,0),(0,1),(1,1),(-1,2),(2,2),(-2,3),(-1,3),(1,3),(-3,4),(-4,5),(0,4),(2,4),(-3,5),(-2,5),(0,5),(3,5)}
    \node at \p [site] {};
\end{scope}
\draw[-latex] (0,-3) -- ++(-1,1);
\draw[-latex] (0,-3) -- ++(0,1);
\draw[-latex] (0,-3) -- ++(1,1);
\end{tikzpicture}\hfill%
\begin{tikzpicture}[animals,baseline=0pt]
\begin{scope}
\clip (-4.4,-.4) rectangle (4.4,5.4);
\draw [grid] (-5.4,0) grid (5.4,6);
\begin{scope}
\clip (-5.4,0) rectangle (5.4,5.4);
\draw [grid,cm={1,1,-1,1,(0,0)}] (-6,-6) grid (6,6);
\draw [grid,cm={1,1,-1,1,(0,1)}] (-6,-6) grid (6,6);
\end{scope}
\segment002
\segment{-1}11
\segment111
\segment{-1}21
\segment223
\segment{-3}32
\segment231
\segment042
\segment341
\segment{-4}54
\segment351
\end{scope}
\draw[-latex] (0,-3) -- ++(-1,0);
\draw[-latex] (0,-3) -- ++(-1,1);
\draw[-latex] (0,-3) -- ++(0,1);
\draw[-latex] (0,-3) -- ++(1,1);
\draw[-latex] (0,-3) -- ++(1,0);
\end{tikzpicture}%
\end{center}
\caption{Directed animals in a selection of lattices. From left to right: the
square lattice, the triangular lattice, the lattice~$\mathcal L_3$, and the
king's lattice.}
\label{fig:lattices}
\end{figure}
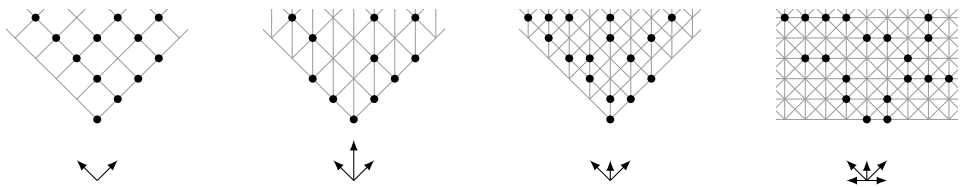

Several techniques have been used to enumerate directed animals. Among them
are direct bijections with other combinatorial objects \cite{gouyou},
comparison with gas models \cite{dhar,bousquet,marckert,albenque} and the use
of Viennot's theory of heaps of pieces
\cite{viennot,betrema,denise,rechnitzer,viennot-multi}. Here, we use the last
method: we describe a bijection between directed animals in the king's lattice
and heaps of segments (the latter are defined in \cite{bousquet-segments}). We
then use this bijection to get enumerative results.

A remarkable feature of all solved lattices is that they seem to belong to the
same universality class: the number of directed animals of area~$n$ is
asymptotic to $\mu^nn^{-1/2}$, where $\mu$ is a lattice-dependant constant.
The average width and height of the directed animals behave similarly in all
solved lattices. Our results show that the king's lattice also belongs to this
universality class.

The paper is organized as follows. In Section~\ref{sec:heaps}, we define heaps
of segments and show their links with directed and multi-directed animals in
the king's lattice. In Section~\ref{sec:directed}, we use these objects to
enumerate directed animals and derive some asymptotic results which prove that
the king's lattice belongs to the same universality class as the square and
triangular lattice. In Section~\ref{sec:multi}, we enumerate multi-directed
animals and show that the generating function of multi-directed animals is not
D-finite. Finally, in Section~\ref{sec:random}, we discuss efficient random
sampling algorithms for our animals.

\section{Animals in the king's lattice and heaps of segments}
\label{sec:heaps}

\subsection{Definitions}

In this section, we define heaps of segments and show their links to directed
animals in the king's lattice. The heaps of segments described here are the
same as in \cite{bousquet-segments}, except that the segment reduced to a
point is not allowed. More information on heaps of pieces in general can be
found in~\cite{viennot}; we recall below some elementary definitions that we
use in this paper.

We call \emph{segment} a closed real interval of the form $[i,j]$, where
$i$ and $j$ are integers such that $j>i$.
Two segments are called \emph{concurrent} if they intersect, even by a point.
A \emph{heap of segments} is a finite sequence of segments, up to commutation
of non-concurrent segments. The \emph{size} of a heap is the sum of the
lengths of the segments composing it.

A heap of segments is represented graphically as in Figure~\ref{fig:heap}. The
segments lying on the ground in the representation of a heap~$H$ are called
\emph{minimal}; they are the segments~$\sigma$ such that $H$ can be written
$\sigma H'$. Similarly, segments~$\sigma$ such that $H$ can be written
$H'\sigma$ are called \emph{maximal}. A heap is a \emph{pyramid} if it has
only one minimal segment.

The set of heaps of segments is equipped with a product. Let $H_1$ and $H_2$
be two heaps. The product $H_1H_2$ is obtained by dropping $H_2$ on top
of~$H_1$. Let $H$ be a heap and $\sigma$ a segment of $H$; there exists a
unique factorization $H=H_1H_2$, where $H_2$ is a pyramid with minimal
segment~$\sigma$. We call it the factorization obtained by \emph{pushing}
$\sigma$. Both operations are illustrated in Figure~\ref{fig:heap}.

\begin{figure}[ht]
\begin{center}
\begin{tikzpicture}[animals,baseline=(b)]
\node (b) at (0,3) {};
\foreach \x in {0,...,7}
    \draw[help lines] (\x,-.5) -- ++(0,6);
\draw[ground] (-.35,-.5) -- ++(7.7,0);
\polymer201
\polymer012
\polymer021
\polymer502
\polymer411
\polymer422
\polymer233
\polymer631
\polymer142
\polymer641
\polymer551
\draw[semithick,rounded corners=.945mm] (-.35,.65) rectangle ++(2.7,.7);
\node at (-1,1) {$\sigma$};
\end{tikzpicture}%
\hspace{6em}%
\begin{tikzpicture}[animals,baseline=(b)]
\node (b) at (0,3) {};
\foreach \x in {0,...,7}
    \draw[help lines] (\x,-.5) -- ++(0,5);
\draw[ground] (-.35,-.5) -- ++(7.7,0);
\polymer201
\polymer502
\polymer411
\polymer422
\polymer631
\polymer641
\begin{scope}[yshift=1cm]
\foreach \x in {0,...,7}
    \draw[help lines] (\x,5.5) -- ++(0,3);
\draw[ground] (-.35,5.5) -- ++(7.7,0);
\polymer062
\polymer071
\polymer273
\polymer182
\polymer581
\draw[semithick,rounded corners=.945mm] (-.35,5.65) rectangle ++(2.7,.7);
%\node at (-1,6) {$\sigma$};
\path (8,0);
\end{scope}
\end{tikzpicture}
\end{center}
\caption{Left: a heap of segments with a distinguished segment~$\sigma$.
Right: the two heaps obtained by pushing the segment~$\sigma$. The product of
these two heaps is equal to the heap on the left.} \label{fig:heap}
\end{figure}
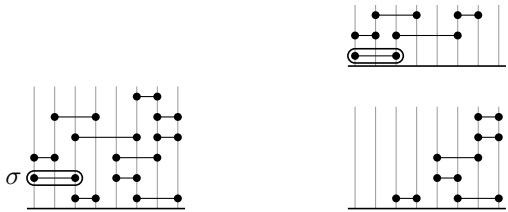

Let $A$ be an animal in the king's lattice. A \emph{segment} of $A$ is a
maximal set of horizontally consecutive sites. If~$S$ is a segment of~$A$,
say $S = \{(i,k),\dotsc,(j-1,k)\}$, we call \emph{projection} of the
segment~$S$ the segment $[i,j]$ and \emph{height} of $S$ the integer~$k$. We
call \emph{projection} of $A$, and we denote by $\pi(A)$, the heap built as
the sequence of the projections of all segments of $A$ in increasing height
order.
This definition is justified by the fact that the projections of two segments
of the animal~$A$ at the same height are necessarily non-concurrent and thus
commute. Examples are shown in Figures~\ref{fig:directed} and~\ref{fig:multi}.

\subsection{Directed animals and pyramids of segments} \label{subsec:directed}

In the next two sections, we describe bijections between animals in the king's
lattice and heaps of segments. We first consider directed animals, defined in
Section~\ref{sec:intro}. In the king's lattice, the source of a directed
animal~$A$ is not unique; it may be any of the bottommost sites. By
convention, we call \emph{source} of~$A$ the leftmost of these sites
(Figure~\ref{fig:directed}, left).

\begin{proposition} \label{prop:directed}
The projection $\pi$ induces a bijection between directed animals and pyramids
of segments, both taken up to a translation.
\end{proposition}

This bijection is illustrated in Figure~\ref{fig:directed}. It works
identically to the classical bijections with heaps of dimers; we refer to
\cite{betrema} for details.

\begin{figure}[ht]
\begin{center}
\begin{tikzpicture}[baseline=(b),animals]
\node (b) at (0,2) {};
\draw[grid,xshift=-.5cm,yshift=-.5cm] (-4,0) grid ++(9,6);
\segment002
\segment{-1}11
\segment111
\segment{-1}21
\segment223
\segment{-3}32
\segment231
\segment042
\segment341
\segment{-4}54
\segment351
\draw[box] (0,0) circle (.4);
\end{tikzpicture}%
%   \begin{tikzpicture}[animals,baseline=0cm]
%   \draw[arrow] (0,2.5) -- node [above] {$\pi$} ++(3,0);
%   \path (-3,0);
%   \path (6,0);
%   \end{tikzpicture}%
$\qquad\xmapsto{\displaystyle\;\pi\;}\qquad$%
\begin{tikzpicture}[baseline=(b),animals]
\node (b) at (0,2) {};
\foreach \x in {-4,...,5}
    \draw[help lines] (\x,-.5) -- ++(0,6);
\draw[ground] (-4.35,-.5) -- ++(9.7,0);
\polymer002
\polymer{-1}11
\polymer111
\polymer{-1}21
\polymer223
\polymer{-3}32
\polymer231
\polymer042
\polymer341
\polymer{-4}54
\polymer351
\end{tikzpicture}%
\end{center}
\caption{Left: a directed animal in the king's lattice (represented, for
clarity, as a set of cells rather than vertices) with its source circled.
Right: the pyramid of segments obtained by replacing each maximal sequence of
$\ell$ consecutive sites by a segment of length~$\ell$. The animal can be
easily reconstructed from the pyramid.}
\label{fig:directed}
\end{figure}
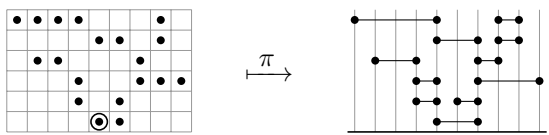

\subsection{Multi-directed animals and connected heaps of segments}
\label{subsec:multi}

Let $A$ be an animal. For any abscissa~$i$, we denote by $b(i)$ the ordinate
of the bottommost site of $A$ at abscissa~$i$ (or $b(i) = +\infty$ if there is
no site of~$A$ at abscissa~$i$). We call \emph{source} of $A$ a site that
realizes a local minimum of~$b$ and \emph{keystone} of $A$ a site that
realizes a local maximum. In case several consecutive sites realize a minimum
or maximum, the source or keystone is the leftmost one
(Figure~\ref{fig:multi}, left). This is an arbitrary choice that does not
alter the definition.

We say that a site~$t$ of~$A$ is \emph{connected} to another site~$s$ if there
exists a directed path from~$s$ to~$t$ visiting only sites of~$A$.

\begin{definition}\rm\label{def:multi}
Let $A$ be an animal. The animal~$A$ is said \emph{multi-directed} if it
satisfies the two conditions:
\begin{itemize}
\item for every site~$t$ of~$A$, there exists a source~$s$ such that $t$ is
connected to~$s$;
\item for every keystone~$t$ of~$A$, there exist two sources~$s_\ell$ and
$s_r$, to the left and to the right of~$t$ respectively, such that $t$ is
connected to both~$s_\ell$ and $s_r$. Moreover, the directed paths connecting
$t$ to $s_\ell$ and $s_r$ do not go through a keystone at the same height
as~$t$.
\end{itemize}
\end{definition}

As a directed animal has only one source and no keystone, every directed
animal is multi-directed. In the following, we say that a heap of segments is
\emph{connected} if it has no empty column (Figure~\ref{fig:multi}).

\begin{proposition} \label{prop:multi}
The projection~$\pi$ induces a bijection between multi-directed animals and
connected heaps of segments, up to a translation.
\end{proposition}

\begin{figure}[ht]
\begin{center}
\begin{tikzpicture}[animals,baseline=(b)]
\node (b) at (0,4) {};
\draw[grid,xshift=-.5cm,yshift=-.5cm] (0,0) grid (13,11);
\node (sl) at (1,1) [site,draw=none] {};
\node (sr) at (8,3) [site,draw=none] {};
\node (t) at (6,6) [site,draw=none] {};
\draw[dark gray] (sl) -- ++(-1,1) -- ++(0,2) -- ++(1,1) -- ++(2,0)
-- ++(1,1) -- (t);
\draw[dark gray] (sr) -- ++(1,1) -- ++(-2,2) -- (t);
\node at (1,0) {$s_\ell$};
\node at (8,2) {$s_r$};
\node at (6,5) {$t$};
\segment021
\segment031
\segment041
\segment112
\segment501
\segment511
\segment322
\segment{11}51
\segment{12}81
\segment{12}61
\segment{12}71
\segment153
\segment467
\segment061
\segment171
\segment283
\segment{11}92
\segment5{10}6
\segment531
\segment851
\segment941
\segment832
\draw[box] (1,1) circle (.4);
\draw[box] (5,0) circle (.4);
\draw[box] (8,3) circle (.4);
\draw[box] (11,5) circle (.4);
\draw[box] (2.6,1.6) rectangle ++(.8,.8);
\draw[box] (5.6,5.6) rectangle ++(.8,.8);
\draw[box] (9.6,5.6) rectangle ++(.8,.8);
\end{tikzpicture}
$\qquad\xmapsto{\displaystyle\;\pi\;}\qquad$%
\begin{tikzpicture}[animals,baseline=(b)]
\node (b) at (0,4) {};
\foreach \x in {0,...,13}
    \draw[help lines] (\x,-.5) -- ++(0,9);
\draw[ground] (-.35,-.5) -- ++(13.7,0);
\polymer011
\polymer021
\polymer031
\polymer102
\polymer501
\polymer511
\polymer322
\polymer{11}01
\polymer{12}11
\polymer{12}21
\polymer{12}31
\polymer143
\polymer457
\polymer051
\polymer161
\polymer273
\polymer{11}62
\polymer586
\polymer531
\polymer821
\polymer911
\polymer802
\end{tikzpicture}%
\end{center}
\caption{Left: a multi-directed animal with four sources (circled) and three
keystones (boxed). The directed paths connecting the keystone~$t$ to the
sources~$s_\ell$ and $s_r$ are shown. Right: the corresponding connected heap
of segments, with has four minimal pieces (one for each source of the
animal).}
\label{fig:multi}
\end{figure}
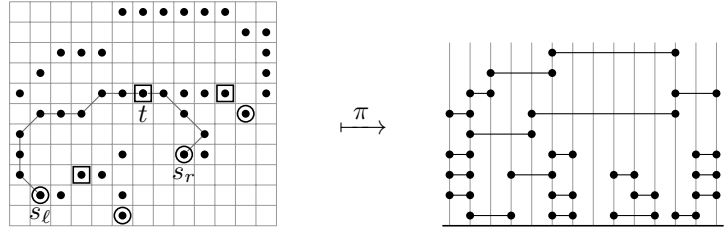

\begin{proof}
Let $A$ be a multi-directed animal. Since $A$ is an animal, the
projection~$\pi(A)$ is a connected heap (Figure~\ref{fig:multi}). We therefore
need to prove that for any connected heap~$H$, there exists a unique
multi-directed animal~$A$ such that $\pi(A) = H$.

Let us call \emph{pre-animal} a finite set of sites with a connected
projection. We define the sources and keystones of a pre-animal in the same
manner as for an animal. A pre-animal is called \emph{multi-directed} if it
satisfies the conditions of Definition~\ref{def:multi}.

We prove by induction the following two statements: for every connected
heap~$H$, there exists a unique (up to a vertical translation) multi-directed
pre-animal~$A$ such that $\pi(A) = H$; moreover, the pre-animal~$A$ is an
animal.

Let $H$ be a connected heap. If $H$ is reduced to a single segment, the result
is obvious. Otherwise, let $\sigma$ be a maximal segment of~$H$. Write $H =
H'\sigma$ and let $H_1,\dotsc,H_k$ be the connected components of the
heap~$H'$, from left to right. Assume that $A$ is a multi-directed pre-animal
such that $\pi(A) = H$. The pre-animal~$A$ is thus composed of a segment~$S$
such that $\pi(S) = \sigma$ and pre-animals $A_1,\dotsc,A_k$ with respective
projections $H_1,\dotsc,H_k$.

As the segment $\sigma$ is maximal, no directed path visiting a site of~$S$
can end in a site not in~$S$. Moreover, all the sources and keystones of the
pre-animals $A_1,\dotsc,A_k$ are also sources and keystones of the
pre-animal~$A$. As $A$ is multi-directed, this forces all the $A_i$'s to be
multi-directed. By the induction hypothesis, all the $A_i$'s are thus uniquely
determined up to a vertical translation and are multi-directed animals. We now
distinguish two cases.
\begin{enumerate}
\item We have $k = 1$. By the first condition of Definition~\ref{def:multi},
the segment~$S$ is connected to a source of~$A_1$. This forces $A_1$ to
touch~$S$ and uniquely determines the pre-animal~$A$; moreover, this means
that $A$ is an animal.
\item We have $k\ge2$. In this case, the segment~$S$ contains $k-1$ keystones,
located between the animals $A_i$ and $A_{i+1}$ for $i = 1,\dotsc,k-1$.
Applying the second condition of Definition~\ref{def:multi} to these keystones
shows that every animal $A_i$, for $i = 1,\dotsc,k$, contains a source $s_i$
such that $S$ is connected to $s_i$. This forces the animals~$A_i$ to touch
the segment~$S$. The pre-animal~$A$ is thus uniquely determined and is an
animal.
\end{enumerate}
The proof is illustrated in Figure~\ref{fig:connected}.
\end{proof}

\begin{figure}[ht]
\begin{center}
\begin{tikzpicture}[animals,baseline=(b)]
\node (b) at (0,3) {};
\foreach \x in {0,...,13}
    \draw[help lines] (\x,-.5) -- ++(0,7);
\draw[ground] (-.35,-.5) -- ++(13.7,0);
\polymer011
\polymer021
\polymer031
\polymer102
\polymer501
\polymer511
\polymer322
\polymer{11}01
\polymer{12}11
\polymer{12}21
\polymer{12}31
\polymer143
\polymer457
\polymer051
\polymer161
\polymer531
\polymer821
\polymer911
\polymer802
\node at (7.5,6) {$\sigma$};
\end{tikzpicture}%
$\qquad\xmapsto{\displaystyle\pi^{-1}}\qquad$%
\begin{tikzpicture}[animals,baseline=(b)]
\node (b) at (0,3) {};
\draw[grid,xshift=-.5cm,yshift=-.5cm] (0,0) grid (13,9);
\segment021
\segment031
\segment041
\segment112
\segment501
\segment511
\segment322
\segment{11}51
\segment{12}81
\segment{12}61
\segment{12}71
\segment153
\segment467
\segment061
\segment171
\segment531
\segment851
\segment941
\segment832
\draw[box] (1,1) circle (.35);
\draw[box] (5,0) circle (.35);
\draw[box] (8,3) circle (.35);
\draw[box] (11,5) circle (.35);
\draw[box] (2.65,1.65) rectangle ++(.7,.7);
\draw[box] (5.65,5.65) rectangle ++(.7,.7);
\draw[box] (9.65,5.65) rectangle ++(.7,.7);
\draw[dark gray] (3.5,5.5) rectangle ++(7,1);
\draw[dark gray] (-.5,-.5) -- ++(6,0) -- ++(0,4) -- ++(-4,4) -- ++(-2,0) --
cycle;
\draw[dark gray] (7.5,2.5) -- ++(2,0) -- ++(0,2) -- ++(-1,1) -- ++(-1,0) --
cycle;
\draw[dark gray] (10.5,4.5) -- ++(2,0) -- ++(0,4) -- ++(-1,0) -- ++(0,-2) --
++(-1,-1) -- cycle;
\node at (7,7.5) {$S$};
\node at (2.5,3.5) {$A_1$};
\node at (8.5,1.5) {$A_2$};
\node at (11.5,3.5) {$A_3$};
\end{tikzpicture}
\end{center}
\caption{Left: a connected heap of segments~$H$ with a maximal
segment~$\sigma$. Right: the multi-directed animal~$A$ such that $\pi(A) = H$.
It is built by recursively building the animals $A_1$, $A_2$ and $A_3$ and
vertically translate them so that they touch the segment~$S$.}
\label{fig:connected}
\end{figure}
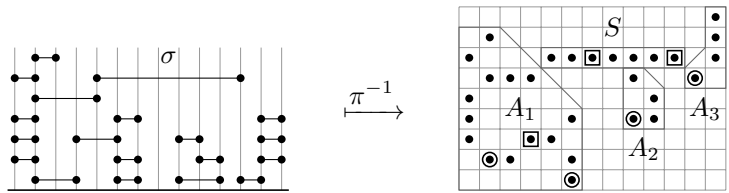

Bousquet-Mélou and Rechnitzer defined multi-directed animals in the square and
triangular lattices and showed that they are in bijection to \emph{connected
heaps of dimers} \cite{rechnitzer}. Applying Definition~\ref{def:multi} to
those lattices leads to a slightly different (but equinumerous) class of
animals. Our definition of multi-directed animals has the advantages of being
more intrinsic and having a vertical symmetry.

\section{Enumeration of directed animals} \label{sec:directed}

\subsection{Exact enumeration}

In this section, we focus on the enumeration of directed animals or,
equivalently, of pyramids of segments. We refine this enumeration by taking
into account another parameter. Let $A$ be a directed animal. Assuming that
the source of~$A$ has abscissa~$0$, we say that $A$ has \emph{left width}~$i$
if the leftmost sites of~$A$ have abscissa~$-i$. We also say that~$A$ is a
\emph{half-animal} if it has left width~$0$. We define similarly the left
width of pyramids of segments and call \emph{half-pyramid} a pyramid with left
width~$0$.

Let $S(t)$ and $D(t)$ be the generating functions of half-animals and
animals, respectively. We also denote by $D(t,u)$ the generating function of
animals where the variable~$u$ tracks the left width.

\begin{theorem}\label{thm:directed}
The generating functions $S(t)$ and $D(t,u)$ satisfy:
\begin{align}\label{S}
S(t) &= \frac{t\bigl(1 + S(t)\bigr)^2}{1 - t\bigl(1 + S(t)\bigr)}\text;\\
\label{D}
D(t,u) &= S(t) + \frac{uS(t)^2}{1 - uR(t)}\text,
\intertext{where the series $R(t)$ is:} \label{R}
R(t) &= S(t) + t\bigl(1 + S(t)\bigr)\text.
\end{align}
In particular, the generating functions of half-animals and animals without
regard for the left width are:
\begin{align}\label{St}
S(t) &= \frac{1 - 3t - \sqrt{1 - 6t + t^2}}{4t}\text;\\\label{Dt}
D(t) &= \frac14\biggl(\frac{1 + t}{\sqrt{1 - 6t + t^2}} - 1\biggr)\text.
\end{align}
\end{theorem}

Interestingly, these generating functions are already known in combinatorics.
Their coefficients are listed respectively as \textbf{A001003} (the little
Schröder numbers) and as \textbf{A047781} in the OEIS~\cite{oeis}. In this
regard, the king's lattice forms a trilogy with the square and triangular
lattices, where the half-animals are enumerated by the Motzkin numbers and the
Catalan numbers, respectively \cite{betrema}.

\begin{proof}
The equations \eqref{S} and \eqref{D} are consequences of the decompositions
of half-pyramids and pyramids illustrated in Figure~\ref{fig:bij}. Let $Q$ be a
half-pyramid and let~$\sigma$ be its base segment (say, $\sigma = [0,\ell]$).
To decompose the pyramid~$Q$, push the lowest segment in column~$0$ other than
the base segment to form the pyramid~$Q_0$ (if that segment does not exist,
let $Q_0$ be the empty heap). Repeat this process in the columns
$1,\dotsc,\ell$. This yields the decomposition $Q = \sigma Q_\ell\dotsm Q_0$,
where $Q_0,\dotsc,Q_\ell$ are possibly empty half-pyramids. The generating
function of such pyramids is thus \smash{$t^\ell(1 + S(t))^{\ell+1}$}. Summing
over all values of~$\ell$ yields \eqref{S}.

Let now $P$ be a pyramid with left width $i > 0$. Let $\sigma$ be the lowest
segment in the column~$-i$. Push~$\sigma$ to form the decomposition $P = P'Q$;
by construction, $Q$ is a half-pyramid. Moreover, let~$i-k$ be the left width
of the pyramid~$P'$ and let~$\ell$ be the length of~$\sigma$. Since $P$ is a
pyramid, we have $1\le k\le\ell$ (Figure~\ref{fig:bij}, below).

With this decomposition, we compute the generating function $D(t,u)$. Pyramids
of left width~$1$ decompose into two half-pyramids and have generating
function~$S(t)^2$. Moreover, to extend a pyramid of left width~$i$ to one of
left width~$i+1$, we either add a new half-pyramid~$Q$ in the column~$-i-1$ or
extend the existing leftmost half-pyramid one unit to the left. These
operations entail multiplying the generating function by $S(t)$ and $t(1 +
S(t))$, respectively. This yields the equation~\eqref{D}.

Finally, solving the equations \eqref{S} and \eqref{D} gives the values
\eqref{St} and \eqref{Dt}.
\end{proof}

\subsection{Bijection with Schröder paths} \label{sec:bijection}

\newcommand\U{\mathbf u}
\newcommand\D{\mathbf d}
\newcommand\F{\mathbf f}

The above results suggest the existence of a bijection between directed
animals in the king's lattice and Schröder paths, especially since directed
animals in the square and triangular lattices are classically in bijection
with Motzkin and Dyck paths, respectively \cite{gouyou}. We present below such
a bijection. We use this bijection for the purpose of random sampling in
Section~\ref{sec:random-directed}.

Consider paths that take three kinds of steps, denoted by $\U$, $\D$ and $\F$
(up, down and flat), with respective coordinates $(1,1)$, $(1,-1)$ and
$(2,0)$. We call \emph{Schröder path} a path that starts and ends at
ordinate~$0$, never visits a negative ordinate, and never takes flat steps at
ordinate~$0$. A path is a \emph{Schröder prefix} if it is the prefix of a
Schröder path. The \emph{length} of a path is the abscissa of its endpoint.

Let $Q$ be a half-pyramid. We define inductively the Schröder path $\psi(Q)$
in the following manner. Write $Q = [0,\ell]Q_\ell\dotsm Q_0$ as in the proof
of Theorem~\ref{thm:directed}. Set:
\begin{equation} \label{f}
\psi(Q) = \psi(Q_0)\,\U\,\psi(Q_1)\,\F\dotsm\F\,\psi(Q_\ell)\text,
\end{equation}
where the empty heap is mapped to the empty path.

Let now $P$ be a pyramid. We define inductively the Schröder prefix~$\phi(P)$.
If $P$ is a half-pyramid, set $\phi(P)$ be $\psi(P)$ minus the last $\D$ step.
If not, decompose it as $P = P'[-i,-i+\ell]Q_\ell\dotsm Q_0$, assuming that
the left width of~$P'$ is $i-k$. Set:
\begin{equation} \label{g}
\phi(P) = \phi(P')\,\U\,\psi(Q_0)\,\F\dotsm\F\,\psi(Q_{k-1})\,\U\,
\psi(Q_k)\,\F\dotsm\F\,\psi(Q_\ell)\text.
\end{equation}
The mappings $\psi$ and $\phi$ are illustrated in Figure~\ref{fig:bij}.

\begin{figure}[ht]
\newcommand\schroeder{ -- ++(.25,1) -- ++(.25,-1) -- ++(.5,2) -- ++(.25,-1)
-- ++(.5,0) -- ++(.25,-1)}
\newcommand\prefix{ -- ++(.25,1) -- ++(.5,0) -- ++(.25,-1) -- ++(.5,2) --
++(.5,0) -- ++(.25,1) -- ++(.25,-1) -- ++(.5,2) -- ++(.5,-2) -- ++(.25,1)}
\newcommand\up{ -- ++(1,1)}
\newcommand\down{ -- ++(1,-1)}
\newcommand\fl{ -- ++(2,0)}
\begin{center}
\begin{tikzpicture}[animals,baseline=(b)]
\node (b) at (0,1) {};
\draw[help lines] (0,5) -- (0,0);
\draw[help lines] (4,1) -- (4,0);
%\draw[latex-latex] (0,-1) -- node [below] {$\ell$} (4,-1);
\draw (0,0) \site -- node [above] {$\sigma$} (4,0) \site;
\begin{scope}[yshift=5cm]
\draw[dark gray] (0,3) -- (0,0) -- (2,0) -- (5,3);
\node at (2,2) {$Q_0$};
\end{scope}
\begin{scope}[yshift=1cm,xshift=4cm]
\draw[dark gray] (0,3) -- (0,0) -- (2,0) -- (5,3);
\node at (2,2) {$Q_\ell$};
\end{scope}
\draw[dotted] (0,5) -- (4,1);
\end{tikzpicture}
$\quad\xmapsto{\displaystyle\;\psi\;}\quad$
\begin{tikzpicture}[animals,baseline=(b)]
\node (b) at (0,1) {};
\draw[dark gray] (0,0) \schroeder (3,1) \schroeder (11,1) \schroeder;
\path (0,0) \site;
\draw (2,0) \site\up\site;
\draw (5,1) \site\fl;
\draw (9,1) \fl\site;
\draw (13,1) \site\down\site;
\node at (1,3) {$\psi(Q_0)$};
\node at (4,5) {$\psi(Q_1)$};
\node at (12,5) {$\psi(Q_\ell)$};
\draw[dotted] (7,1) -- (9,1);
\end{tikzpicture}\\[1cm]
\begin{tikzpicture}[animals,baseline=(b)]
\node (b) at (0,3) {};
\begin{scope}[xshift=-4cm,yshift=3cm]
\draw[help lines] (0,5) -- (0,-4);
\draw[help lines] (4,1) -- (4,0);
\draw[help lines] (2,0) -- (2,-4);
\draw[thin,latex-latex] (0,-4) -- node [below] {$k$} (2,-4);
\draw (0,0) \site -- node [above] {$\sigma$} (4,0) \site;
\begin{scope}[yshift=5cm]
\draw[dark gray] (0,3) -- (0,0) -- (2,0) -- (5,3);
\node at (2,2) {$Q_0$};
\end{scope}
\begin{scope}[yshift=1cm,xshift=4cm]
\draw[dark gray] (0,3) -- (0,0) -- (2,0) -- (5,3);
\node at (2,2) {$Q_\ell$};
\end{scope}
\draw[dotted] (0,5) -- (4,1);
\end{scope}
\draw[dark gray] (0,0) -- ++(-2,2) -- ++(0,1);
\draw[dark gray] (2,0) -- ++(3,3);
\draw (0,0) \site -- (2,0) \site;
\node at (1,2) {$P'$};
\end{tikzpicture}
$\quad\xmapsto{\displaystyle\;\phi\;}\quad$
\begin{tikzpicture}[animals,baseline=(b)]
\node (b) at (0,3) {};
\draw[dark gray] (.25,0)\prefix (5,4)\schroeder (13,4)\schroeder
(16,5)\schroeder (24,5)\schroeder;
\draw (.25,0)\site (4,3)\site\up\site (7,4)\site\fl
(11,4)\fl\site (15,4)\site\up\site (18,5)\site\fl
(22,5)\fl\site (26,5)\site;
\draw[dotted] (9,4) -- ++(2,0);
\draw[dotted] (20,5) -- ++(2,0);
\node at (2,5) {$\phi(P')$};
\node at (6,7) {$\psi(Q_0)$};
\node at (14,7) {$\psi(Q_{k-1})$};
\node at (17,9) {$\psi(Q_k)$};
\node at (25,9) {$\psi(Q_\ell)$};
\end{tikzpicture}
\end{center}
\caption{Above: a half-pyramid decomposed into the form $\sigma P_\ell\dotsm
P_0$ and its image by the bijection~$\psi$. Below: a pyramid of positive left
width decomposed into the form $P'\sigma Q_\ell\dotsm Q_0$ and its image by
the bijection~$\phi$.}
\label{fig:bij}
\end{figure}
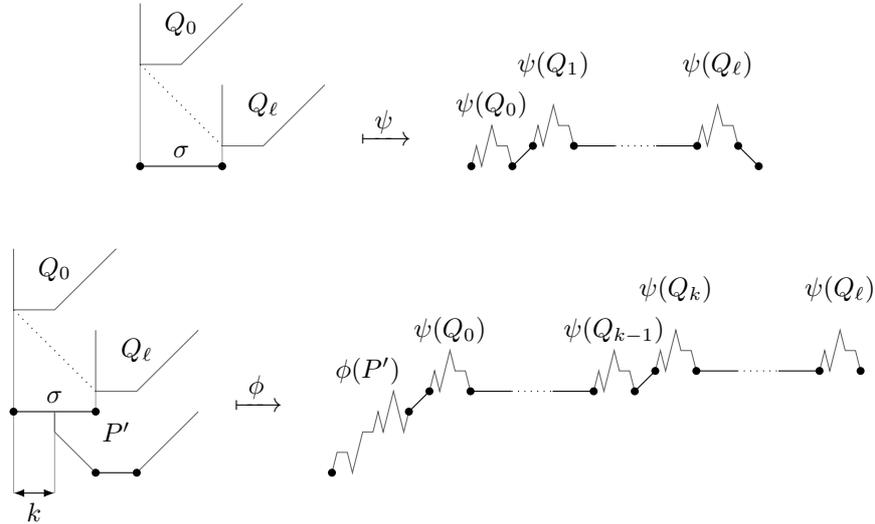

\begin{theorem} \label{thm:bijection}
The mapping $\psi$ is a bijection between half-pyramids of segments of
size~$n$ and Schröder paths of length~$2n$. The mapping $\phi$ is a bijection
between pyramids of segments of size~$n$ and Schröder prefixes of
length~$2n-1$.
\end{theorem}

Note that, unlike in the square and triangular lattices, the left width of a
pyramid does not map to the final height of its corresponding Schröder prefix.
This means that, for dealing with the left width, working on heaps rather than
paths is usually more convenient.

\begin{proof}
First, we readily prove by induction that the mappings~$\psi$ and $\phi$ map
pyramids of size~$n$ to paths of length~$2n$ and $2n-1$, respectively. We now
prove that they are bijections by describing inductively their inverse
mappings.

Let $\omega$ be a Schröder path. Consider the last visit of $\omega$ at
height~$0$ apart from its endpoint. Cut the path~$\omega$ at that point and
at all flat steps at height~$1$ after it. This leads to the canonical
decomposition:
\begin{equation} \label{fi}
\omega = \omega_0\,\U\,\omega_1\,\F\dotsm\F\,\omega_\ell\,\D\text,
\end{equation}
where $\omega_0,\dotsc,\omega_\ell$ are Schröder paths. Let
$P_0,\dotsc,P_\ell$ be the inverse images of the paths
$\omega_0,\dotsc,\omega_\ell$, found inductively. The inverse image
of~$\omega$ is then $[0,\ell]P_\ell\dotsm P_0$, where each $P_i$ is translated
$i$ units to the right.

Now, let $\eta$ be a Schröder prefix of odd length. Let $h$ be its final
height, which is also odd. If $h = 1$, we form $\phi^{-1}(\eta)$ as
$\psi^{-1}(\eta\,\D)$. If $h\ge3$, we cut the path~$\eta$ at its last visit at
height~$h-2$, at the flat steps at height~$h-1$ after it, at its last visit at
height~$h-1$ and at the flat steps at height~$h$ after it. This yields the
canonical decomposition:
\begin{equation} \label{gi}
\eta = \eta'\,\U\,\omega_0\,\F\dotsm\F\,\omega_{k-1}\,\U\,
\omega_k\,\F\dotsm\F\,\omega_\ell\text,
\end{equation}
where $\omega_0,\dotsc,\omega_\ell$ are Schröder paths. Let
$P' = \phi^{-1}(\eta')$, found inductively, and let
$Q_i = \psi^{-1}(\omega_i)$ for $i = 0,\dotsc,\ell$. We form the inverse image
of~$P$ as $P'\sigma Q_\ell\dotsm Q_0$, where $\sigma$ is a segment of
length~$\ell$ located such that the left widths of $P$ and $P'$ differ by~$k$.

By the definitions \eqref{f} and \eqref{g}, these operations are indeed the
inverse mappings of $\psi$ and $\phi$.
\end{proof}

\subsection{Asymptotics}

In this section, we derive from Theorem~\ref{thm:directed} results dealing
with the asymptotic behavior of the numbers of directed animals.

\begin{theorem} \label{thm:directed-asympt}
Let $n\ge1$. Let $d(n)$ be the number of directed animals of area~$n$ and let
$\lw(n)$ be the average left width of the same animals. As $n$ tends to
infinity, we have the following estimates:
\begin{align*}
d(n)&\sim2^{-7/4}\frac{\bigl(3 + \sqrt8\bigr)^n}{\sqrt{\pi n}}\text;\\
\lw(n)&\sim2^{-3/4}\sqrt{\pi n}.
\end{align*}
\end{theorem}

This behavior -- $d(n)$ being asymptotically of the form $\mu^n/\sqrt n$ and
$lw(n)$ of the order of $\sqrt n$ -- is the same as in the square and
triangular lattices (where the growth constants $\mu$ are $3$ and $4$,
respectively; see \cite{gouyou}). Another parameter of interest is the
\emph{height} of the animals; using the random generator of
Section~\ref{sec:random-directed}, we estimate that the average height of the
directed animals of area~$n$ behaves like:
\[h(n)\sim\lambda n^\gamma\text,\]
where $\gamma = 0.8177...$. This is also in line with the square and
triangular lattices \cite{longitudinal}.

\begin{proof}
The exact values $d(n)$ and $\lw(n)$ are obtained from the generating function
$D(t,u)$ as follows:
\begin{align*}
d(n) &= [t^n] D(t,1)\text;\\
\lw(n) &= \frac1{d(n)}[t^n]\diff Du(t, 1)\text.
\end{align*}
Both generating functions have radius of convergence $\rho = 3 - \sqrt8$.
They have two singularities, at $t = 3\pm\sqrt8$, and admit an analytic
continuation in the domain \smash{$\mathbb
C\setminus\bigl[3-\sqrt8,3+\sqrt8\bigr]$}. As $t$ tends to~$\rho$, we have the
following estimates:
\begin{align*}
D(t,1) &= \frac{2^{-7/4}}{\sqrt{1 - t/\rho}}
+ \bigo(1)\text;\\
% + \smallo\biggl(\frac1{\sqrt{1 - t/\rho}}\biggr)\text;\\
\diff Du(t, 1) &= \frac{2^{-5/2}}{1 - t/\rho}
+ \bigo\biggl(\frac1{\sqrt{1 - t/\rho}}\biggr)\text.
% + \smallo\biggl(\frac1{1 - t/\rho}\biggr)\text.
\end{align*}
The results follow using classical singularity analysis
\cite[Theorem~VI.4]{flajolet}.
\end{proof}

\section{Enumeration of multi-directed animals} \label{sec:multi}

In this section, we enumerate multi-directed animals or, equivalently,
connected heaps of segments. The tool we use here is an adaptation of the
\emph{Nordic decomposition}, which was invented by Viennot
\cite{viennot-multi} to enumerate connected heaps of dimers in a combinatorial
way.

\subsection{Nordic decomposition}

Let $C$ be a connected heap of segments that is not a pyramid. Let $\sigma$
be the rightmost minimal segment of~$C$; let $C = C'P$ be the factorization
obtained by pushing~$\sigma$. Let $C_1\dotsm C_n$ be the decomposition of~$C'$
in connected components, from left to right; let $H$ be the heap $C_2\dotsm
C_n$.

As the heap~$C$ is known up to a translation, we assume that the rightmost
point of the heap~$C_1$ has abscissa~$-1$. This fixes the segment~$\sigma$,
say $\sigma = [k,\ell]$. We call \emph{Nordic decomposition} of the heap~$C$
the quadruple $(C_1, k, H, P)$. This decomposition is illustrated in
Figure~\ref{fig:nordic}.

\begin{figure}[ht]
\begin{center}
\begin{tikzpicture}[animals,baseline=0cm]
\draw[dark gray] (-6,0) rectangle ++(6,4);
\node at (-3,2) {$C$};
\end{tikzpicture}%
\equals%
\begin{tikzpicture}[animals,baseline=0cm]
\draw (10,0) node (1) [site] {} -- ++(2,0) node (2) [site] {};
\draw[dark gray] (1) -- ++(-3,3) (2) -- ++(3,3);
\node at (11,2) {$P$};
\end{tikzpicture}%
\plus%
\begin{tikzpicture}[animals,baseline=0cm]
\draw[dark gray] (-6,0) rectangle ++(6,4);
\node at (-3,2) {$C_1$};
\draw[dark gray]
(1,0) rectangle ++(3,3) (5,0) rectangle ++(1,1) (7,0) rectangle ++(2,2);
\node at (5.5,2) {$H$};
\draw (10,3) node (1) [site] {} -- node [below] {$\sigma$} ++(2,0) node (2) [site] {};
\draw[dark gray] (1) -- ++(-10,2) -- ++(-1,1) (2) -- ++(3,3);
\node at (7.5,5) {$P$};
% \foreach \x in {1,...,9}
% \draw[help lines] (\x,-.5) -- ++(0,-1);
% \draw[semithick,decorate,decoration=brace]
% (9,-2) -- node [below=1pt] {$i$} ++(-8,0);
\draw [help lines] (0,0) -- ++(0,-1);
%\node at (0,-1.5) {$0$};
\draw [help lines] (1) -- (10,-1);
%\node at (10,-1) {$i+1$};
\draw[latex-latex] (0,-1) -- node [below] {$k+1$} (10,-1);
\end{tikzpicture}
\end{center}
\caption{The Nordic decomposition of a non-pyramid connected heap: pushing the
rightmost minimal segment~$\sigma$ yields the pyramid~$P$. The heap~$C_1$ is
the leftmost connected component of the remaining heap. The other components
compose the heap~$H$, which lives in the gap of width~$k+1$ between the
heap~$C_1$ and the segment~$\sigma$.} \label{fig:nordic}
\end{figure}
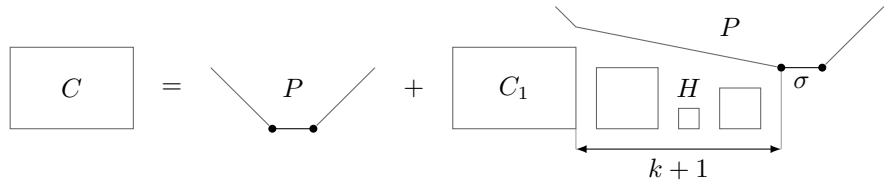

In the following, we call \emph{$k$-heap} a heap with all segments included in
$[0,k-1]$. Unlike other heaps, we do not identify $k$-heaps that differ by a
translation.

\pagebreak[2]

\begin{proposition} \label{prop:nordic}
The Nordic decomposition is a bijection between non-pyramid connected heaps
and quadruples of the form~$(C_1, k, H, P)$ such that:
\begin{itemize}
\item $C_1$ is a connected heap;
\item $k$ is a non-negative integer;
\item $H$ is a $k$-heap;
\item $P$ is a pyramid with left width greater than~$k$.
\end{itemize}
\end{proposition}

\begin{proof}
First, let $C$ be a non-pyramid connected heap and let $(C_1,k,H,P)$ be its
Nordic decomposition. We first show that this quadruple satisfies the
conditions of the lemma. The heap~$C_1$ is connected by definition. Moreover,
as the component~$C_1$ is not concurrent to~$\sigma$, we have $k\ge0$.
Furthermore, as the components $C_2,\dotsc,C_n$ are concurrent neither
to~$C_1$ nor to~$\sigma$, all the segments of the heap~$H$ are included
in~$[0,k-1]$. Finally, to be connected to the heap~$C_1$, the pyramid~$P$ must
have left width at least~$k+1$.

To conclude, we show that it is possible to recover the heap~$C$ from its
Nordic decomposition. To do that, we translate the heap~$C_1$ so that
its rightmost point has abscissa~$-1$ and the pyramid $P$ so that its minimal
segment is of the form~$[k,\ell]$. The heap~$C$ is then equal to the product
$C_1HP$.
\end{proof}

\subsection{Exact enumeration}

The Nordic decomposition enables us to establish the following theorem, which
enumerates multi-directed animals.

\begin{theorem} \label{thm:multi}
Let $M = M(t)$ be the generating function of multi-directed animals. Let $S =
S(t)$, $D = D(t,1)$ and $R = R(t)$ be the power series defined by \eqref{S},
\eqref{D} and \eqref{R}, respectively. Moreover, let $Q = Q(t)$ be the power
series defined by:
\begin{equation} \label{Q} Q(t) = (2 - 2t)S(t) - t. \end{equation}
The generating function $M$ is given by:
\begin{equation} \label{M} M = \frac{D}{\displaystyle1 -
\sum\limits_{k\ge0}S(1 + S)^k\dfrac{QR^k}{1 - QR^k}}. \end{equation}
\end{theorem}

To prove the theorem, we first establish the following lemma.

\begin{lemma} \label{lem:HD}
Let $k\ge0$; let $H_k(t)$ be the generating function of $k$-heaps and let
$D_{>k}(t)$ be the generating function of pyramids of segments with left width
greater than~$k$. We have the following identity:
\[H_k(t)D_{>k}(t) = S(1+S)^k\frac{QR^k}{1 - QR^k}.\]
\end{lemma}

\pagebreak[2]

\begin{proof}
Let $\mathcal A_k$ be the set of heaps of segments~$A$ satisfying the
conditions:
\begin{itemize}
\item every segment of~$A$ is included in $[0,+\infty)$;
\item the rightmost minimal segment of~$A$ is of the form~$[k,\ell]$.
\end{itemize}
Let $A_k(t)$ be the generating function of the set~$\mathcal A_k$.

Let $A$ be a heap of~$\mathcal A_k$ and let $\sigma$ be the rightmost minimal
segment of~$A$. We decompose the heap~$A$ in two ways. First, we push the
lowest segment in the columns~$0,\dotsc,k-1$ in succession. This yields $k$
possibly empty half-pyramids; what remains is a half-pyramid with minimal
segment~$\sigma$. This decomposition yields the identity:
\[A_k(t) = S(t)\bigl(1 + S(t)\bigr)^k.\]

Alternatively, push the segment~$\sigma$. This yields a pyramid with left
width at most~$k$; what remains is a $k$-heap. We deduce the second identity:
\[A_k(t) = H_k(t)\bigl(D(t) - D_{>k}(t)\bigr).\]
Equating the two above expressions for $A_k(t)$, we find:
\begin{equation} \label{HD}
H_k(t)D_{>k}(t) = S(t)\bigl(1 + S(t)\bigr)^k
\frac{D_{>k}(t)}{D(t) - D_{>k}(t)}.
\end{equation}

Finally, the identity~\eqref{D} shows that the generating function of pyramids
with left width~$i$ is $D_i(t) = S(t)^2R(t)^{i-1}$ for~$i\ge1$. Therefore, we
have:
\[D_{>k}(t) = S(t)^2\frac{R(t)^k}{1 - R(t)}.\]
An elementary computation from the definitions \eqref{D}, \eqref{R} and
\eqref{Q} of $D$, $R$ and $Q$ lets us rewrite this into:
\[D_{>k}(t) = D(t) Q(t) R(t)^k.\]
Plugging this into the right hand side of \eqref{HD} establishes the lemma.
\end{proof}

This lemma is illustrated in Figure~\ref{fig:QR}.

\begin{figure}[ht]
\begin{center}
\begin{tikzpicture}[animals,baseline=0pt]
\draw[dark gray] (5,0) -- ++(0,3) (7,0) -- ++(3,3);
\node at (7,2) {$S$};
\draw[dark gray] (0,6) -- ++(0,-3) -- ++(6,0) -- ++(3,3);
\node at (4,5) {$(1 + S)^k$};
\draw[help lines] (0,3) -- ++(0,-4);
\node at (0,-2) {$0$};
\draw[help lines] (5,0) -- ++(0,-1);
\node at (5,-2) {$k$};
\draw (5,0) node [site] {} -- ++(2,0) node [site] {};
\end{tikzpicture}\hfil%
% \begin{tikzpicture}[animals,baseline=0pt]
% \draw(0,-3) -- ++(0,10);
% \end{tikzpicture}\hfil%
\begin{tikzpicture}[animals,baseline=0pt]
\draw[dark gray] (0,0) rectangle (4,3);
\node at (2,1.5) {$H_k$};
\draw[dark gray] (5,3) -- ++(-3,2) -- ++(0,1);
\draw[dark gray] (7,3) -- ++(3,3);
\node at (5.5,5) {$D - D_{>k}$};
\draw[help lines] (0,-1) -- ++(0,1);
\node at (0,-2) {$0$};
\draw[help lines] (5,-1) -- ++(0,4);
\node at (5,-2) {$k$};
\draw (5,3) node [site] {} -- ++(2,0) node [site] {};
\end{tikzpicture}
\end{center}
\caption{The two decompositions of the objects of~$\mathcal A_k$ involved in
the proof of Lemma~\ref{lem:HD}.}
\label{fig:QR}
\end{figure}
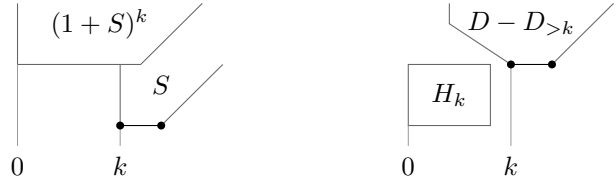

\begin{proof}[Proof of Theorem~\ref{thm:multi}]
The Nordic decomposition (Proposition~\ref{prop:nordic}) gives the following
functional equation, with the notations of Lemma~\ref{lem:HD}:
\[M(t) = D(t) + \sum_{k\ge0} M(t) H_k(t) D_{>k}(t),\]
which is equivalent to:
\[M(t) = \frac{D(t)}{\displaystyle1 - \sum\limits_{k\ge0} H_k(t) D_{>k}(t)}.\]
We conclude using Lemma~\ref{lem:HD}.
\end{proof}

\subsection{Asymptotics and nature of the series} \label{subsec:asympt-multi}

In this section, we describe the asymptotic behavior of the number of directed
and multi-directed animals of area~$n$ as $n$ tends to infinity. We also study
the average width of these animals. We also show that the generating function
of multi-directed animals is not D-finite.

\begin{theorem} \label{thm:multi-asympt}
The generating function~$M(t)$ of multi-directed animals has a unique,
dominant simple pole at $\rho_M = 0.154...$. Moreover, the series~$M(t)$ is
not D-finite.
\end{theorem}

A similar result exists in the square and triangular lattices
\cite{rechnitzer}. To prove this theorem, we again establish a lemma.

\begin{lemma} \label{lem:B}
Let $B = B(t)$ be the power series:
\[B = \sum_{k\ge0}S(1 + S)^k\frac{QR^k}{1 - QR^k}.\]
The series~$B(t)$ has a unique, dominant simple pole at $\rho_B$, satisfying:
\begin{equation}\label{rhoB}
1 - 5\rho_B - 7\rho_B^2 + \rho_B^3 = 0.
\end{equation}
Moreover, the series~$B$ is not D-finite.
\end{lemma}

\begin{proof}
We start by rewriting the expression of~$B$ into:
\begin{equation*}
B = \sum_{k\ge0}\Biggl(\sum_{j\ge1}S(1+S)^kQ^jR^{jk}\Biggr)
  = \sum_{j\ge1} \frac{SQ^j}{1 - (1 + S)R^j}.
\end{equation*}
We use this form to locate the singularities of the power series~$B$ in the
interval $[0,\rho)$, where $\rho = 3-\sqrt8$ is the radius of convergence of
the generating functions $Q$, $R$ and $S$.

First, we note that $Q(0) = R(0) = 0$ and $Q(\rho) = R(\rho) = 1$. Since $Q$
and $R$ have positive coefficients, this means that $Q(x) < 1$ and $R(x) < 1$
for $x$ in $[0,\rho)$. Thus, for any fixed~$x$, the $j$th summand in the above
expression and all its derivatives decrease exponentially as $j$ tends to
infinity. This shows that the infinite sum does not create any singularities.
Therefore, the only singularities of~$B$ in the interval~$[0,\rho)$ are
points~$x$ such that $(1 + S(x))R(x)^j = 1$ for some~$j\ge1$.

Let $j\ge1$ and consider the function $f_j$ defined for $x$ in $[0,\rho)$ by:
\begin{align*}
f_j(x) = \bigl(1 + S(x)\bigr)R(x)^j.
\end{align*}
As $S$ is increasing and satisfies $S(\rho) = 1/\sqrt2$, this function
reaches~$1$ at a unique point that we denote by~$\rho_j$. Moreover, since
$R(x) < 1$, we have $f_i(x) > f_j(x)$ for all~$0 < x < \rho$ and $i < j$.
Therefore, the sequence $(\rho_j)$ is increasing.

This proves that $B$ is not D-finite, since it has infinitely many
singularities. Moreover, as $B$ has positive real coefficients, Prigsheim's
Theorem \cite[Theorem~IV.6]{flajolet} shows that the radius of convergence
of~$B$ is $\rho_1$. In other words, we have:
\[\bigl(1 + S(\rho_B)\bigr) R(\rho_B) = 1.\]
Equation \eqref{rhoB} follows by performing elimination with the
definitions \eqref{S} and \eqref{R} of $S$ and $R$.
\end{proof}

\begin{proof}[Proof of Theorem~\ref{thm:multi-asympt}]
To prove the theorem, we rewrite the equation~\eqref{M} into:
\[M(t) = \frac{D(t)}{1 - B(t)}.\]
We then use Lemma~\ref{lem:B}; since $\rho_B$ is a simple pole of~$B$, we have
$B(t)\to+\infty$ as $t\to\rho_B$. Therefore, the value $B(t)$ reaches $1$ at a
unique point in $[0,\rho_B)$, which is a simple pole. Again, Pringsheim's
Theorem shows that this point is the radius of convergence of~$M$. Numerical
estimates yield the announced value of~$\rho_M$. Finally, since $D$ is
D-finite and $B$ is not, $M$ is not D-finite.
\end{proof}

\begin{corollary} \label{cor:asympt}
Let $m(n)$ be the number of multi-directed animals of area~$n$. As $n$ tends
to infinity, this number satisfies:
\[m(n)\sim\lambda\mu^n\text,\]
with $\mu = 1/\rho_M = 6.475...$.
Moreover, the average number of sources and the average width of the
multi-directed animals of area~$n$ grow linearly with~$n$.
\end{corollary}

\begin{proof}
This result stems from the fact that the class of multi-directed animals
follows a \emph{supercritical sequence schema} as a consequence of
Theorem~\ref{thm:multi-asympt}. In particular, \cite[Theorem~V.1]{flajolet}
gives the estimate of $m(n)$ and shows that to decompose a connected heap of
size~$n$ all the way to a pyramid, a linear number of Nordic decompositions is
needed on average. Since every Nordic decomposition adds at least one minimal
segment and one to the width of the heap, the average number of sources and
the average width of the multi-animals are also linear.
\end{proof}

\section{Random sampling} \label{sec:random}

The question of the random sampling of directed animals (finding an efficient
algorithm that outputs a uniformly distributed animal with a given area~$n$)
has attracted some attention; the most efficient known algorithm is given by
Barcucci, Pinzani and Sprugnoli \cite{barcucci}. This algorithm is based on
the bijection with Motzkin paths found in \cite{gouyou} and outputs a directed
animal of area~$n$ in average time~$\bigo(n)$.

In this section, we show how this algorithm can be adapted to the king's
lattice, using the bijection of Section~\ref{sec:bijection} and an algorithm
of Penaud, Pergola, Pinzani and Roques \cite{penaud} to sample Schröder paths.
We also propose an algorithm for the random sampling of multi-directed animals
using a very different method, namely a Boltzmann sampler \cite{boltzmann}.
The reasons for these choices are discussed in Section~\ref{sec:multi-random}.

Both algorithms sample an animal of area~$n$ in average time $\bigo(n)$.
Examples of their outputs are shown in Figure~\ref{fig:random}.

\subsection{Random sampling of directed animals} \label{sec:random-directed}

Using the bijection of Section~\ref{sec:bijection}, we can obtain an algorithm
for the generation of directed animals: we sample a Schröder prefix of the
desired length, then apply the bijection~$\phi^{-1}$ to get a pyramid of
segments, then apply the bijection of Section~\ref{subsec:directed} to get a
directed animal. We show that all these operations take linear time.

\paragraph{Sample a Schröder prefix.}
The problem of the random generation of Schröder paths has been studied in
\cite{penaud}. The algorithm to sample a Schröder prefix of length $2n-1$ is
as follows. Let $p = \sqrt2-1$ be the positive solution of the equation
$2p+p^2 = 1$. Set $\eta$ to the empty path and repeat:
\begin{itemize}
\item add to $\eta$ a step $\U$, $\D$ or $\F$ with respective probabilities
$p$, $p$ and $p^2$;
\item if $\eta$ is not a Schröder prefix, discard~$\eta$ and start over;
\item if $\eta$ has length $2n-1$, output $\eta$;
\item if $\eta$ has length $2n$, discard $\eta$ and start over.
\end{itemize}

It is shown in \cite{penaud} that this algorithm outputs a uniform Schröder
prefix in average time $\mathcal O(n)$, since it needs an average of $\mathcal
O(\sqrt n)$ trials each costing $\mathcal O(\sqrt n)$ on average. Further
details on the complexity can be found in \cite{louchard}.

\paragraph{From Schröder prefix to pyramid of segments.}
We now show that the bijections $\psi^{-1}$ and $\phi^{-1}$ can be computed in
linear time. Let us start with~$\psi$.

Given a Schröder walk~$\omega$, set $H\leftarrow\varnothing$, $j\leftarrow0$
and start with an empty stack. For every step $s$ of $\omega$, do:
\begin{itemize}
\item if $s = \U$, push $j$ on the stack and then set $j\leftarrow j+1$;
\item if $s = \F$, set $j\leftarrow j+1$;
\item if $s = \D$, pop $i$ from the stack, set $H\leftarrow[i,j]\,H$ and then
$j\leftarrow i$.
\end{itemize}
Finally, return the heap~$H$.

\begin{lemma}
The above algorithm computes $\psi^{-1}(\omega)$ in linear time.
\end{lemma}

\begin{proof}
We prove the following stronger fact: assume that the loop is entered with
initial value $j = i$. Then the algorithm performs the operation $H\leftarrow
H'\,H$, where $H'$ is $\psi^{-1}(\omega)$ translated $i$ units to the right.
The values of $j$ and the stack are unchanged.

We prove this by induction on the length of~$\omega$. If $\omega$ is empty,
the result is obvious; otherwise, decompose $\omega$ as in \eqref{fi}. By
induction hypothesis, the algorithm first reads $\omega_0$ and adds
$\psi^{-1}(\omega_0)$ translated by~$i$ to~$H$. It then reads a $\U$ step,
pushes~$i$ on the stack and sets $j$ to $i+1$. Again by induction hypothesis,
it reads the words~$\omega_1,\dotsc,\omega_\ell$ and adds to~$H$ their images
by~$\psi^{-1}$ translated by $i+1,\dotsc,i+\ell$ (since every $\F$ step
increments the variable~$j$). Finally, it reads a $\D$ step, at which point
the value of~$j$ is $i+\ell$, and so adds the segment $[i,i+\ell]$ to~$H$. In
total, the algorithm indeed added to $H$ the pyramid $\psi^{-1}(\omega)$
translated by~$i$. Moreover, the values of~$j$ and the stack are reset to
their initial values.

Finally, it is obvious to see that the algorithm runs in linear time.
\end{proof}

Given a Schröder prefix~$\eta$, computing the pyramid~$\phi^{-1}(\eta)$ is
then easy: we read the word~$\eta$ in reverse order and find the factors
$\omega_0,\dotsc,\omega_\ell$ of the decomposition \eqref{gi}. We compute
their images by~$\psi^{-1}$ using the above algorithm. Finally, we iterate
this procedure to compute the pyramid $\phi^{-1}(\eta_0)$. This constructs
$\phi^{-1}(\eta)$ in linear time.

\paragraph{From pyramid of segments to directed animal.}
A linear algorithm that turns a pyramid of dimers into a directed animal in
the square lattice is given in~\cite{betrema}; this algorithm readily adapts
to our case (we refer to that paper for details). Given a pyramid~$H$, we
assign a height to all segments of~$H$ by maintaining an array that, for each
column~$i$, contains the maximal height of a segment of~$H$ in the column~$i$.
This enables us to compute the corresponding directed animal~$A$ in linear
time, with only $\mathcal O(\sqrt n)$ integers of extra space on average.

\subsection{Random sampling of multi-directed animals}
\label{sec:multi-random}

The method presented above does not work on multi-directed animals, as we were
not able to find bijections with appropriate families of lattices walks.
Instead we use the \emph{Boltzmann samplers} presented in \cite{boltzmann}. A
Boltzmann sampler for a combinatorial class~$\mathcal A$ with a generating
function $A(t)$ outputs an element~$a$ of~$\mathcal A$ with probability
\[\mathbb P(a) = \frac{x^{\abs a}}{A(x)},\]
where $x$ is a parameter chosen depending on the desired output length. Of
course, $x$ must be lower than the radius of convergence of the power
series~$A(t)$.

A useful feature of Boltzmann samplers is that, from samplers for two classes
$\mathcal A$ and $\mathcal B$, we can easily build samplers for the disjoint
union $\mathcal A+\mathcal B$, the Cartesian product $\mathcal A\times\mathcal
B$, and the sequence class $\seq(\mathcal A)$.
Moreover, given a sampler for the class $\mathcal A\times\mathcal B$, we can
build a sampler for the class $\mathcal A$ by sampling an ordered pair
$(a,b)$ and outputting~$a$. A Boltzmann sampler can thus be automatically
constructed from a \emph{combinatorial specification}
\cite[Chapter~1]{flajolet}.

The final ingredient that we use is the \emph{critical sampler} described in
\cite[Section~7.1]{boltzmann}. Critical Boltzmann samplers are very efficient
algorithms to sample a class described as a supercritical sequence. As shown
in Section~\ref{subsec:asympt-multi}, this is the case of multi-directed
animals.

Of course, the general framework of Boltzmann samplers applies just as well to
directed animals. However, the class of directed animals cannot be written as
a supercritical sequence, as evidenced by the fact that their dominant
singularity is not a simple pole. Therefore, the critical sampler referred to
above is inapplicable. The average complexity of a Boltzmann sampler for
directed animals would be $\bigo(n^2)$, which is less effective than the
algorithm of the previous section; however, \emph{approximate size} sampling
could be done in linear time. See \cite{boltzmann} for all the details.

Note that the algorithm described below can be easily adapted to the square
and triangular cases, where, to our knowledge, no random sampling algorithms
previously existed for multi-directed animals.

\medskip

Let $\mathcal S$ and $\mathcal D$ be the classes of half-animals and directed
animals, respectively. The decompositions in the proof of
Theorem~\ref{thm:directed}, illustrated in Figure~\ref{fig:bij}, amount to
specifications for the combinatorial classes $\mathcal S$ and $\mathcal D$:
\begin{align} \label{sS}
\mathcal S &= \mathcal Z\times(1+\mathcal S)^2\times
\seq\bigl(\mathcal Z\times(1+\mathcal S)\bigr)\text;\\ \label{sD}
\mathcal D &= \mathcal S + \mathcal S^2\times\seq(\mathcal R)\text,
\end{align}
where $\mathcal Z$ is the atomic class and the class $\mathcal R$ is defined
by:
\begin{equation} \label{sR}
\mathcal R = \mathcal S + \mathcal Z\times(1 + \mathcal S)\text.
\end{equation}
From this, we automatically build a Boltzmann sampler for the class~$\mathcal
D$.

Moreover, the Nordic decomposition (Proposition~\ref{prop:nordic}) gives:
\begin{equation} \label{sM}
\mathcal M = \mathcal D\times\seq(\mathcal B),
\end{equation}
where the class $\mathcal B$ is defined, with obvious notations, as:
\begin{equation} \label{sB}
\mathcal B = \sum_{k\ge0}\mathcal H_k\times\mathcal D_{>k}\text.
\end{equation}
To find specifications for the class $\mathcal H_k$, we use the construction
of the proof of Lemma~\ref{lem:HD}, which yields two specifications for the
class~$\mathcal A_k$:
\begin{align} \label{sAk}
\mathcal A_k &= \mathcal S\times(1 + \mathcal S)^k\text;\\ \label{sHk}
\mathcal A_k &= \mathcal H_k\times\mathcal D_{\le k}\text.
\end{align}
Moreover, we have, using again the proof of Theorem~\ref{thm:directed},
\begin{equation} \label{sDk}
\mathcal D_{>k} = \mathcal S^2\times\mathcal R^k\times\seq(\mathcal R).
\end{equation}
The specification \eqref{sAk} yields a Boltzmann sampler for the
class~$\mathcal A_k$. Using this sampler and the specification \eqref{sHk}, we
find a Boltzmann sampler for the class~$\mathcal H_k$. The specification
\eqref{sDk} gives a Boltzmann sampler for the class~$\mathcal D_{>k}$.
Finally, using both these samplers and the specification \eqref{sB},
we obtain a Boltzmann sampler for the class~$\mathcal B$.

To obtain a sampler for the class~$\mathcal M$, we use the algorithm of
\cite[Section~7.1]{boltzmann} with the specification \eqref{sM}, setting the
parameter of our Boltzmann samplers to the value $x = \rho_M$ (see
Section~\ref{subsec:asympt-multi}). Let $n$ be the area of the animal to be
sampled. We start by sampling an element of~$\mathcal D$ and then add elements
of~$\mathcal B$ until the area of the animal reaches or exceeds~$n$. If it
reaches exactly~$n$, we output the result, otherwise we reject it and start
over. Theorem~7.1 of \cite{boltzmann} shows that this algorithm is correct and
runs, on average, in time~$\bigo(n)$.

The same algorithm can also be made to produce an approximate size output,
allowing it to succeed in one trial with high probability; the correct version
of this procedure appears in \cite[Chapter~1]{pivoteau}.

\begin{figure}[ht]
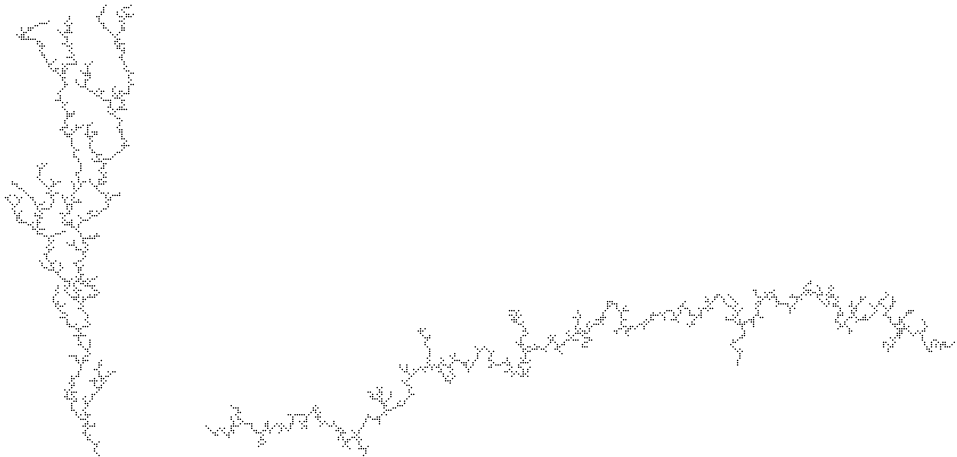

\input fig-random.tex
\caption{Random directed (left) and multi-directed (right) animals of
area~$1000$.} \label{fig:random}
\end{figure}

\bibliographystyle{abbrv}

\bibliography{biblio}{}
\end{document}